\documentclass[oribibl,envcountsame]{llncs}
\usepackage{amssymb}
\usepackage{amsmath}
\usepackage{booktabs}
\usepackage{url}
\usepackage{color}
\usepackage{epic}
\newcommand{\Fp}{\mathbb{F}_p}
\newcommand{\Fq}{\mathbb{F}_q}

\newcommand{\Fqk}{\mathbb{F}_{q^k}}
\newcommand{\ZetaCq}{Z(C/\mathbb{F}_q;T)}

\newcommand{\JCq}{J(C/\mathbb{F}_q)}
\newcommand{\JCp}{J(C/\mathbb{F}_p)}
\newcommand{\JCtq}{J(\tilde{C}/\mathbb{F}_q)}
\newcommand{\JCtp}{J(\tilde{C}/\mathbb{F}_p)}

\newcommand{\diffop}{{\bf\Delta}}
\newcommand{\stirling}[2]{\left\{\begin{array}{c}#1\\#2\end{array}\right\}}
\newcommand{\identity}{1_{\scriptscriptstyle{G}}}
\newtheorem{algorithm}[theorem]{Algorithm}

\begin{document}

\title{Computing $L$-series of hyperelliptic curves}
\author{Kiran S. Kedlaya\thanks{Kedlaya was supported by 
NSF CAREER grant DMS-0545904 and a Sloan Research Fellowship.}
 \and Andrew V. Sutherland}
\institute{Department of Mathematics \\ Massachusetts Institute of Technology \\
77 Massachusetts Avenue \\ Cambridge, MA 02139 \\
\email{(kedlaya|drew)@math.mit.edu}}

\maketitle

\begin{abstract}
We discuss the computation of coefficients of the $L$-series associated to a hyperelliptic curve over $\mathbb{Q}$ of genus at most 3, using point counting, generic group algorithms, and $p$-adic methods.
\end{abstract}

\section{Introduction}

For $C$ a smooth projective curve of genus $g$ defined over $\mathbb{Q}$, the $L$-function $L(C,s)$ is conjecturally (and provably for $g=1$) an entire function containing much arithmetic information about $C$.  Most notably, according to the conjecture of Birch and Swinnerton-Dyer, the order of vanishing of $L(C,s)$ at $s=1$ equals the rank of the group $J(C/\mathbb{Q})$ of rational points on the Jacobian of $C$.

It is thus natural to ask to what extent we are able to compute with the $L$-function. This splits into two subproblems:
\begin{enumerate}
\item[1.] For appropriate $N$, compute the first $N$ coefficients of the Dirichlet series expansion $L(C,s) = \prod_p L_p(p^{-s})^{-1} = \sum_{n=1}^\infty c_n n^{-s}$.
\vspace{3pt}
\item[2.] From the Dirichlet series, compute $L(C,s)$ at various values of $s$ to suitable numerical accuracy. (The Dirichlet series converges for $\mathrm{Real}(s) > 3/2$.)
\end{enumerate}
In this paper, we address problem 1 for hyperelliptic curves of genus $g \leq 3$ with a distinguished rational Weierstrass point. This includes in particular the case of elliptic curves, and indeed we have something new to say in this case; we can handle significantly larger coefficient ranges than other existing implementations.  We say nothing about problem 2; we refer instead to 
\cite{Dokchitser:ComputingSpecialValues}.

Our methods combine efficient point enumeration with generic group algorithms as discussed in the second author's PhD thesis \cite{Sutherland:Thesis}. 
For $g > 2$, we also apply $p$-adic cohomological methods, as introduced by the 
first author \cite{Kedlaya:PointCounting} and refined by 
Harvey \cite{Harvey:LargeCharacteristicKedlaya}.  Since what
we need is adequately described in these papers, we focus our presentation on the point counting and generic group techniques and use an existing $p$-adic cohomological implementation provided by Harvey.
(The asymptotically superior Schoof-Pila method  \cite{Schoof:ECPointCounting2,Pila:PointCounting} only becomes practically better far beyond the ranges we can hope to handle.)

As a sample application, we compare statistics for Frobenius eigenvalues of particular curves to theoretical predictions.  These include the Sato-Tate conjecture for $g=1$, and appropriate analogues in the Katz-Sarnak framework for $g>1$; for the latter, we find little prior numerical evidence in the literature.

\section{The Problem}
Let $C$ be a smooth projective curve over $\mathbb{Q}$ of genus $g$ .
We wish to determine the polynomial $L_p(T)$ appearing in $L(C,s) = \prod L_p(p^{-s})^{-1}$, for $p \le N$.  We consider only $p$ for which $C$ is defined and nonsingular over $\Fp$ (almost all of them), referring to \cite{Silverman:EllipticCurves2,Deninger:LFunctions} in the case of bad reduction.  The polynomial $L_q(T)$ appears as the numerator of the local zeta function
\begin{equation}\label{equation:ZetaDefinition}
\ZetaCq = \exp\left(\sum_{k=1}^{\infty}N_kT^k/k\right) = \frac{L_q(T)}{(1-T)(1-qT)},
\end{equation}
where $N_k$ counts the points on $C$ over $\Fqk$.  Here $q$ is any prime power, however we are primarily concerned with $q=p$ an odd prime.  The rationality of $\ZetaCq$ is part of the well known theorem of Weil \cite{Weil:ZetaFunction}, which also requires
\begin{equation}
L_q(T) = \sum_{i=0}^{2g}a_iT^i
\end{equation}
to have integer coefficients satisfying $a_0 = 1$ and $a_{2g-i} = p^{g-i}a_i,$ for $0 \le i < g$.  To determine $L_q(T)$, it suffices to compute $a_1,\ldots,a_g$.

For reasons of computational efficiency we restrict ourselves to curves which may be described by an affine equation of the form $y^2 = f(x)$, where $f(x)$ is a monic polynomial of degree $d = 2g+1$ (hyperelliptic curves with a distinguished
rational Weierstrass point).
We denote by $\JCq$ the group of $\Fq$-rational points on the Jacobian variety of $C$ over $\Fq$ (the \emph{Jacobian} of $C$ over $\Fq$), and use $\JCtq$ to denote the Jacobian of the quadratic twist of $C$ over $\Fq$.

\vspace{6pt}
\noindent
We consider three approaches to determining $L_p(T)$ for $g \le 3$:
\begin{enumerate}
\item
{\bf Point counting}: Compute $N_1$,\ldots,$N_g$ of (\ref{equation:ZetaDefinition}) by enumerating the points on $C$ over $\Fp, \mathbb{F}_{p^2},\ldots, \mathbb{F}_{p^g}$.  The coefficients $a_1,\ldots,a_g$ can then be readily derived  from (\ref{equation:ZetaDefinition}) \cite[p. 135]{Cohen:HECHECC}.  This requires $O(p^g)$ field operations.
\vspace{6pt}
\item
{\bf Group computation}: Use generic algorithms to compute $L_p(1) = \#\JCp$, and, for $g > 1$, compute $L_p(-1) = \#\JCtp$.  Then use $L_p(1)$ and $L_p(-1)$ to determine $L_p(T)$  \cite[Lemma 4]{Sutherland:Jacobian}. This involves a total of $O(p^{(2g-1)/4})$ group operations.
\vspace{6pt}
\item
{\bf $p$-adic methods}:  Apply extensions of Kedlaya's algorithm \cite{Kedlaya:PointCounting,Harvey:LargeCharacteristicKedlaya} to compute (modulo $p$) the characteristic polynomial $\chi(T) = T^{2g}L_p(T^{-1})$ of the Frobenius endomorphism on $\JCp$, then use generic algorithms to compute the exact coefficients of $L_p(T)$.  The asymptotic complexity is $\tilde{O}(p^{1/2})$.\footnote{For fixed $g \geq 4$, one works modulo $p^{\lfloor g/2 - 1 \rfloor}$ to obtain the same complexity.}
\end{enumerate}

Computing the coefficients of $L_p(T)$ for all $p\le N$ necessarily requires time and space exponential in $\lg N$, since the output contains $\Theta(N)$ bits.  In practice, we are limited to $N$ of moderate size: on the order of $2^{40}$ in genus 1, $2^{28}$ in genus 2, and $2^{26}$ in genus 3 (larger in parallel computations).  We expect to compute $L_p(T)$ for a large number of relatively small values of $p$.  Constant factors will have considerable impact, however we first consider the asymptotic situation.

The $O(p^g)$ complexity of point counting makes it an impractical method to compute $a_1$, \ldots, $a_g$ unless $p$ is very small.  However, point counting over $\Fp$ is an efficient way to compute $a_1 = N_1 - p - 1$ for a reasonably large range of $p$ when $g > 1$, requiring only $O(p)$ field operations.  Knowledge of $a_1$ aids the computation of $\#J(C/\Fp)$, reducing the complexity of the baby-steps giant-steps search to $O(p^{1/4})$ in genus 2 and $O(p)$ in genus 3.
The optimal strategy then varies (c.f. \cite[pp. 32-33]{Elkies:CurvesOverFiniteFields}), according to genus and range of $p$:
\vspace{-12pt}
\subsubsection{Genus 1}
The $O(p^{1/4})$ complexity of generic group computation makes it the compelling choice, easily outperforming point counting for $p>2^{10}$.
\vspace{-12pt}
\subsubsection{Genus 2}
There are three alternatives: (i) $O(p)$ field operations followed by $O(p^{1/2})$ group operations, (ii) $O(p^{3/4})$ group operations, or (iii) an $\tilde{O}(p^{1/2})$ $p$-adic computation.  We find the range in which (iii) becomes optimal to be past the feasible values of $N$.
\vspace{-12pt}
\subsubsection{Genus 3}
The choice is between (i) $O(p)$ field operations followed by $O(p)$ group operations and (ii) an $\tilde{O}(p^{1/2})$ $p$-adic computation followed by $O(p^{1/4})$ group operations.  Here the $p$-adic algorithm plays the major role once $p>2^{15}$.

\section{Point Counting}
Counting points on $C$ over $\Fp$ plays a key role in our strategy for genus 2 and 3 curves.  Moreover, it is a useful tool in its own right.  If one wishes to study the distribution of $\#J(C/\Fp) = L_p(1)$, or to simply estimate $L_p(p^{-s})$, the value $a_1$ may be all that is required.

Given $C$ in the form $y^2 = f(x)$, the simplest approach is to build a table of the quadratic residues in $\Fp$ (typically stored as a bit-vector), then evaluate $f(x)$ for all $x\in\Fp$.  If $f(x) = 0$, there is a single point on the curve, and otherwise either two points (if $f(x)$ is a residue) or none.  Additionally, we add a single point at infinity (recall that  $f$ has odd degree).  A not-too-na\"\i ve implementation computes the table of quadratic residues by squaring half the field elements, then uses $d$ field multiplications and $d$ field additions for each evaluation of $f(x)$, where $d$ is the degree of $f$.  A better approach uses finite differences, requiring only $d$ field additions (subtractions) to compute each $f(x)$.

Let $f(x) = \sum f_j x^j$ be a degree $d$ polynomial over a commutative ring $R$.  Fix a nonzero $\delta\in R$ and define the linear operator $\diffop$ on $R[x]$ by 
\begin{equation}
(\diffop f)(x) = f(x+\delta) - f(x).
\end{equation}
For any $x_0\in R[x]$, given $f(x_0)$, we may enumerate the values $f(x_0+n\delta)$ via
\begin{equation}\label{equation:FiniteDifferences}
f(x_0+(n+1)\delta) = f(x_0+n\delta) + \diffop f(x_0+n\delta).
\end{equation}
To enumerate $f(x_0+n\delta)$ it suffices to enumerate $\diffop f(x_0+n\delta)$, which we also do via (\ref{equation:FiniteDifferences}), replacing $f$ with $\diffop f$.  Since $\diffop^{d+1}f$ = 0, each step requires only $d$ additions in $R$, starting from the initial values $\diffop^kf(x_0)$ for $0\le k \le d$.

When $R = \Fp$, this process enumerates $f(x)$ over the entire field and we simply set $\delta = 1$ and $x_0 = 0$.
As subtraction modulo $p$ is typically faster than addition, instead of (\ref{equation:FiniteDifferences}) we use
\begin{equation}
f(x_0+(n+1)\delta) = f(x_0+n\delta) - (-\diffop f)(x_0+n\delta).
\end{equation}
The necessary initial values are then $(-1)^k\diffop^k f(0)$.

\vspace{6pt}
\begin{algorithm}[Point Counting over $\Fp$]\label{algorithm:PointCounting}
Given a polynomial $f(x)$ over $\Fp$ of odd degree $d$ and a vector $M$ identifying nonzero quadratic residues in $\Fp$:
\vspace{3pt}
\rm
\renewcommand\labelenumi{\theenumi.}
\begin{enumerate}
\item
Set $t_k \leftarrow (-1)^k\diffop^kf(0)$, for $0\le k \le d$, and set $N \leftarrow 1$.
\item
\vspace{6pt}
For $i$ from 1 to $p$:
\begin{enumerate}
\vspace{3pt}
\item
If $t_0 = 0$, set $N \leftarrow N + 1$, and if $M[t_0]$, set $N\leftarrow N + 2$.
\item
Set $t_0 \leftarrow t_0 - t_1$, $t_1 \leftarrow t_1 - t_2$, \ldots, and $t_{d-1} \leftarrow t_{d-1} - t_d$.
\end{enumerate}
\end{enumerate}
\it
\vspace{3pt}
Output $N$.
\end{algorithm}
\vspace{6pt}

The computation $t_k = t_k - t_{k+1}$ is performed using integer subtraction, adding $p$ if the result is negative.
The map $M$ is computed by enumerating the polynomial $f(x) = x^2$ for $x$ from 1 to $(p-1)/2$ and setting $M[f(x)] = 1$, using a total of $p$ subtractions (and no multiplications).

The size of $M$ may be cut in half by only storing residues less than $p/2$.  One then uses $M[\min(t_0,p-t_0)]$, inverting $M[p-t_0]$ when $p\equiv 3 \bmod 4$.  This slows down the algorithm, but is worth doing if $M$ exceeds the size of cache memory.

It remains only to compute $\diffop^kf(0)$.  We find that
\begin{equation}\label{equation:StirlingIdentity}
\diffop^k f(0) = \sum_j k!\stirling{j}{k} f_j = \sum_j T_{j,k} f_j,
\end{equation}
where the bracketed coefficient denotes a Stirling number of the second kind.  The triangle of values $T_{j,k}$ is represented by sequence A019538 in the OEIS \cite{Sloane:OEIS}. Since (\ref{equation:StirlingIdentity}) does not depend on $p$, it is computed just once for each $k\le d$.

In the process of enumerating $f(x)$, we can also enumerate $f(x)+g(x)$ with $e+1$ additional field subtractions, where $e$ is the degree of $g(x)$.  The case where $g(x)$ is a small constant is particularly efficient, since nearby entries in $M$ are used.  The last two columns in Table \ref{table:PointCounting} show the amortized cost per point of applying this approach to the curves $y^2 = f(x)$, $f(x)+1$, \ldots, $f(x)+31$.

\begin{table}
\begin{center}
\begin{tabular}{@{}lcrrcrrcrr@{}}
&  &\multicolumn{2}{c}{Polynomial}& &\multicolumn{2}{c}{Finite}& &\multicolumn{2}{c}{Finite}\\
&  &\multicolumn{2}{c}{Evaluation}& &\multicolumn{2}{c}{Differences}& &\multicolumn{2}{c}{Differences $\times 32$}\\
\cmidrule(r){3-4}\cmidrule(r){6-7}\cmidrule(r){9-10}
$p\approx$& & Genus 2 & \hspace{6pt}Genus 3 & & Genus 2 & \hspace{6pt}Genus 3 & & Genus 2 & \hspace{6pt}Genus 3\\
\midrule
$2^{16}$&\hspace{12pt}	& 195.1 & 257.2 &\hspace{18pt} & 6.1  &  7.8 &\hspace{18pt} & 1.1 & 1.1\\
$2^{17}$&	& 196.3 & 262.6 & & 6.0  &  6.9 & & 1.1 & 1.1\\
$2^{18}$&	& 192.4 & 259.8 & & 6.0  &  6.8 & & 1.1 & 1.1\\
$2^{19}$&	& 186.3 & 251.1 & & 6.0  &  6.8 & & 1.1 & 1.1\\
$2^{20}$&	& 187.3 & 244.1 & & 7.2  &  8.0 & & 1.1 & 1.3\\
$2^{21}$&	& 172.3 & 240.8 & & 8.8  &  9.4 & & 1.2 & 1.3\\
\midrule
$2^{22}$&	& 197.9 & 233.9 & & 12.1 & 13.4 & & 1.2 & 1.3\\
$2^{23}$&	& 229.2 & 285.8 & & 12.8 & 14.6 & & 2.6 & 2.7\\
$2^{24}$&	& 258.1 & 331.8 & & 41.2 & 44.0 & & 3.5 & 4.7\\
\midrule
$2^{25}$&	& 304.8 & 350.4 & & 53.6 & 55.7 & & 4.8 & 4.9\\
$2^{26}$&	& 308.0 & 366.9 & & 65.4 & 67.8 & & 4.8 & 4.6\\
$2^{27}$&	& 318.4 & 376.8 & & 70.5 & 73.1 & & 4.9 & 5.0\\
$2^{28}$&	& 332.2 & 387.8 & & 74.6 & 76.5 & & 5.1 & 5.2\\
\bottomrule
\vspace{2pt}
\end{tabular}
\caption{Point counting $y^2 = f(x)$ over $\Fp$ (CPU nanoseconds/point)}\label{table:PointCounting}
\end{center}
\vspace{-12pt}
\begin{minipage}{1.0\linewidth}
\small
The middle rows of Table \ref{table:PointCounting} show the transition of $M$ from $L2$ cache to general memory.  The top section of the table is the most relevant for the algorithms considered here, as asymptotically superior methods are used for larger $p$.
\normalsize
\end{minipage}
\end{table}

\section{Group Computations}\label{section:GroupComputations}

The performance of generic group algorithms is typically determined by two quantities: the time required to perform a group operation, and the number of operations performed.  We briefly mention two techniques that reduce the former, then consider the latter in more detail.
\vspace{-6pt}
\subsection{Faster Black Boxes}

The performance of the underlying finite field operations used to implement the group law on the Jacobian can be substantially improved using a Montgomery representation to perform arithmetic modulo $p$ \cite{Montgomery:ModularArithmetic}.  Another optimization due to Montgomery that is especially useful for the algorithms considered here is the simultaneous inversion of field elements (see \cite[Alg. 11.15]{Cohen:HECHECC}).\footnote{This algorithm can be applied to any group.}  With an affine representation of the Jacobian each group operation requires a field inversion, but uses fewer multiplications than alternative representations.  To ameliorate the high cost of field inversions, we then modify our algorithms to perform group operations ``in parallel".

In the baby-steps giant-steps algorithm, for example, we fix a small constant $n$, compute $n$ ``babies" $\beta$, $\beta^2$, \ldots, $\beta^n$, then march them in parallel using steps of size $n$ (the giant steps are handled similarly).  In each parallel step we execute $n$ group operations to the point where a field inverse is required, perform all the field inversions together for a cost of $3n-3$ multiplications and one inversion, then use the results to complete the group operations.  Exponentiation can also benefit from parallelization, albeit to a lesser extent.

These two optimizations are most effective when applied in combination, as may be seen in Table \ref{table:BlackBox}.

\begin{table}
\begin{center}
\begin{tabular}{@{}llcrrrcrrr@{}}
& & &\multicolumn{3}{c}{Standard}&\hspace{18pt} &\multicolumn{3}{c}{Montgomery}\\
\cmidrule(r){4-6}\cmidrule(r){8-10}
$g$\hspace{6pt} & $\hspace{12pt}p$&\hspace{12pt} & $\times 1$ & \hspace{6pt}$\times 10$ & \hspace{6pt}$\times 100$ &
& \hspace{6pt}$\times 1$ & \hspace{6pt}$\times 10$ & \hspace{6pt}$\times 100$ \\
\midrule
1 & $2^{20}+7$  & &  501 &  245 &  215 &&  239 &   89 &   69\\
1 & $2^{25}+35$ & &  592 &  255 &  216 &&  286 &   93 &   69\\
1 & $2^{30}+3$  & &  683 &  264 &  217 &&  333 &   98 &   69\\
\midrule
2 & $2^{20}+7$  & & 1178 &  933 &  902 &&  362 &  216 &  196\\
2 & $2^{25}+35$ & & 1269 &  942 &  900 &&  409 &  220 &  197\\
2 & $2^{30}+3$  & & 1357 &  949 &  902 &&  455 &  225 &  196\\
\midrule
3 & $2^{20}+7$  & & 2804 & 2556 & 2526 &&  642 &  498 &  478\\
3 & $2^{25}+35$ & & 2896 & 2562 & 2528 &&  690 &  502 &  476\\
3 & $2^{30}+3$  & & 2986 & 2574 & 2526 &&  736 &  506 &  478\\
\bottomrule
\vspace{2pt}
\end{tabular}
\caption{Black box performance (CPU nanoseconds/group operation)}\label{table:BlackBox}
\end{center}
\vspace{-12pt}
\begin{minipage}{1.0\linewidth}
\small
The heading $\times n$ indicates $n$ group operations performed ``in parallel".  All times are for a single thread of execution. 
\normalsize
\end{minipage}
\end{table}

\subsection{Generic Order Computations}

Our approach to computing $\#\JCq = L_q(1)$ is based on a generic algorithm to compute the structure of an arbitrary abelian group \cite{Sutherland:Thesis}.  We are aided both by absolute bounds on $L_q(1)$ derived from the Weil conjectures (theorems), as well as predictions regarding its distribution within these bounds based on a generalized form of the Sato-Tate conjecture (proven for most genus 1 curves over $\mathbb{Q}$ in \cite{Harris:SatoTateProof}).  We first consider the general algorithm.

We assume we have a black box for an abelian group $G$ (written multiplicatively) that can generate uniformly random group elements.  For Jacobians, these can be obtained via decompression techniques \cite[14.1-2]{Cohen:HECHECC}.\footnote{This becomes costly when $g > 2$, where we use the simpler approach of \cite[p. 307]{Cohen:HECHECC}.}  We also suppose we are given bounds $M_0$ and $M_1$ such that $M_0\le |G| \le M_1$.

The first (typically only) step is to compute the group exponent, $\lambda(G)$, the least common multiple of the orders of all the elements of $G$.  This is accomplished by initially setting $E = 1$, and for a random $\alpha\in G$, computing the order of $\beta = \alpha^E$ using a baby-steps giant-steps search on the interval $[M_0/E,M_1/E]$.  We then update $E\leftarrow|\beta|E$ and repeat the process until either (1) there is only one multiple of $E$ in the interval $[M_0,M_1]$, or (2) we have generated $c$ random elements, where $c$ is a confidence parameter.  In the former case we must have $|G| = E$,
and in the latter case $E = \lambda(G)$, with probability greater than $1-2^{2-c}$ \cite[Proposition 8.3]{Sutherland:Thesis}.  For large Jacobians, (1) almost always applies, however for the relatively small groups considered here, (2) arises more often, particularly when $g > 1$.  Fortunately, this does not present undue difficulty.

\begin{proposition}\label{proposition:GroupOrder}
Given $\lambda(G)$ and $M_0$ such that $M_0\le |G| < 2M_0$, the value of $|G|$ can be computed using $O(|G|^{1/4})$ group operations.
\end{proposition}
\begin{proof}[sketch]
The bounds on $|G|$ imply that it is enough to know the order of all but one of the $p$-Sylow subgroups of $G$ (the $p$ dividing $|G|$ are obtained from $\lambda(G)$).  Following Algorithm 9.1 of \cite{Sutherland:Thesis}, we use $\lambda(G)$ to compute the order of each $p$-Sylow subgroup $H\subseteq G$ using $O(|H|^{1/2})$ group operations; however, we abandon the computation for any $p$-Sylow subgroup that proves to be larger than $\sqrt{|G|}$.  This can happen at most once, and the remaining successful computations uniquely determine $|G|$ within the interval $[M_0,2M_0)$. $\Box$
\end{proof}

From the Weil interval (see (\ref{equation:WeilInterval}) in section \ref{subsection:BSGS2}) we find that $M_1 < 2M_0$ for all $q > 300$ and $g\le 3$.
Proposition \ref{proposition:GroupOrder} implies that group structure computations will not impact the complexity of our task.  Indeed, computing $\#J(C/\Fq)$ is almost always dominated by the first computation of $|\beta|$.

Given $\beta\in G$ and the knowledge that the interval $[M_0,M_1]$ contains an integer $M$ for which $\beta^M = \identity$, a baby-steps giant-steps search may be used to find such an $M$.  This is not necessarily the order of $\beta$, it is a multiple of it.  We can then factor $M$ and compute $|\beta|$ using $\tilde{O}(\lg M)$ group operations \cite[Ch. 7]{Sutherland:Thesis}.  The time to factor $M$ is negligible in genus 2 and 3 (compared to the group computations), and in genus 1 we note that if a sieve is used to enumerate the primes up to $N$, the factorization of every $M$ in the interval $[M_0,M_1]$ can be obtained at essentially no additional cost, using $O(\sqrt{N})$ bytes of memory.

An alternative approach avoids the computation of $|\beta|$ from $M$ by attempting to prove that $M$ is the only multiple of $|\beta|$ in the interval.  Write $[M_0,M_1]$ as $[C-R,C+R]$, and suppose the search to find $M=C\pm r$ has shown $\beta^{n} \ne \identity$ for all n in $(C-r,C+r)$.  If $M$ is not the only multiple of $|\beta|$ in $[C-R,C+R]$, then $|\beta|$ is a divisor of $M$ satisfying $2r \le |\beta|\le R+r$.
In particular, if $P$ is the largest prime factor of $M$ and $P > R+r$ and $M/P < 2r$, then $M$ must be unique.  When $R = O(M^{1/2})$ this happens fairly often (about half the time).  When it does not happen, one can avoid an $\tilde{O}(\lg M)$ order computation at the cost of $O(R^{1/2})$ group operations by searching the remainder of the interval \emph{on the opposite side of $M$}.  This is only worthwhile when $R$ is quite small, but can be helpful in genus 1.\footnote{These ideas were sparked by a conversation with Mark Watkins, who also credits Geoff Bailey.}

\vspace{-6pt}
\subsection{Optimized Baby-Steps Giant-Steps in the Jacobian - Part I}

The Mumford representation of $\JCq$ uniquely represents a reduced divisor of the curve $y^2 = f(x)$ by a pair of polynomials $(u,v)$.  The polynomial $u$ is monic, with degree at most $g$, 
and divides $v^2-f$ \cite[p. 307]{Cohen:HECHECC}.  The inverse of $(u,v)$ is simply $(u,-v)$, which makes two facts immediate:
\begin{enumerate}
\item
The cost of group inversions is effectively zero.
\item
The element $(u,v)$ has order 2 if and only if $v=0$ and $u$ divides $f$.
\end{enumerate}
Fact 1 allows us to apply the usual optimization for fast inverses \cite[p. 250]{Cohen:CANT}, reducing the number of group operations by a factor of $\sqrt{2}$ (we no longer count inversions).  Fact 2 gives us a bijection between the 2-torsion subgroup of $\JCq$ and polynomials dividing $f$ of degree at most $g$ (exactly half the polynomials dividing $f$).  If $k$ counts the irreducible polynomials in the unique factorization of $f$, then the 2-rank of $\JCq$ is $k-1$ and $2^{k-1}$ divides $\#\JCq$.\footnote{Computing $k$ requires only a distinct-degree factorization of $f$, see \cite[Alg. 3.4.3]{Cohen:CANT}.}

When $k > 1$, we start with $E = 2^{k-1}$ in our computation of $\lambda(G)$ above, reducing the number of group operations by a factor of $2^{(k-1)/2}$.  Otherwise, we know $\#\JCq$ is odd and can reduce the number of group operations by a factor of $\sqrt{2}$.  The total expected benefit of fast inversions and knowledge of 2-rank is at least a factor of 2.10 in genus 1, 2.31 in genus 2, and 2.48 in genus 3.
\vspace{-6pt}
\subsection{Optimized Baby-Steps Giant-Steps in the Jacobian - Part II}\label{subsection:BSGS2}

We come now to the most interesting class of optimizations, those based on the distribution of $\#\JCq$.  The Riemann hypothesis for curves (proven by Weil) states that $L_q(T)$ has roots lying on a circle of radius $q^{-1/2}$ about the origin of the complex plane.  As $L_q(T)$ is a real polynomial of even degree with $L_q(0) = 1$, these roots may be grouped into conjugate pairs.

\begin{definition}
A \em{unitary symplectic polynomial} $p(z)$ is a real polynomial of even degree
with roots $\alpha_1,...\alpha_g,\bar{\alpha}_1,...\bar{\alpha}_g$ all on the unit circle.
\end{definition}

The unitary symplectic polynomials are precisely those arising as the characteristic polynomial of a unitary symplectic matrix.  The Riemann hypothesis for curves implies that $p(z) = L_q(zq^{-1/2})$ is a unitary symplectic polynomial.
The coefficients of $p(z) = \sum a_jz^j$ may be bounded by
\begin{equation}\label{equation:abounds}
|a_j| \le \binom{2g}{j}.
\end{equation}
The corresponding bounds on the coefficients of $L_q(T)$ constrain the value of $L_q(1) = \#\JCq$, yielding the Weil interval
\begin{equation}\label{equation:WeilInterval}
(\sqrt{q}-1)^{2g} \le \#\JCq \le (\sqrt{q}+1)^{2g}.
\end{equation}
For the $a_j$ with $j$ odd, the well known bounds in (\ref{equation:abounds}) are tight, however for even $j$ they are not.  We are particularly interested in the coefficient $a_2$.

\begin{proposition}\label{proposition:a2Bounds}
Let $p(z) = \sum a_jz^j$ be a unitary symplectic polynomial of degree $2g$.  For fixed $a_1$, $a_2$ is bounded by an interval of radius at most $g$.  In fact
\begin{align}\label{equation:a2UpperBound}
a_2 &\le g + \left(\frac{g-1}{2g}\right)a_1^2;\\\label{equation:a2LowerBound}
a_2 &\ge -g + 2 + \left(a_1^2-\delta^2\right)/2.
\end{align}
The value $\delta\le 2$ is the distance from $a_1$ to the nearest integer congruent to $0\bmod 4$ (when $g$ is odd), or $2\bmod 4$ (when $g$ is even).
\end{proposition}
\begin{proof}
Define $\beta_j = \alpha_j + \bar{\alpha}_j$ for $1\le j \le g$, where the $\alpha_j$ are the roots of $p(z)$.  Then $a_1 = \sum \beta_j$ and $a_2 = g + (a_1^2-t_2)/2$, where $t_2 = \sum \beta_j^2$.  For fixed $a_1$, $t_2$ is minimized by $\beta_j = a_1/g$, yielding (\ref{equation:a2UpperBound}), and $t_2$ is maximized by $\beta_j = \pm 2$ for $j<g$ and $\beta_g = \delta$, yielding (\ref{equation:a2LowerBound}) (note that $|\beta_j|\le 2$).  The proposition follows. $\Box$
\end{proof}

We have as a corollary, independent of $a_1$, the bound $a_2 \ge -g$, and for $g$ odd, $a_2 \ge 2-g$.  In genus 2, the proposition reduces to Lemma 1 of \cite{Matsuo:BabyStepPointCounting}, however we are especially interested in the genus 3 case, where our estimate of $a_2$ will determine the leading constant factor in the time to compute $\#\JCq$.  In genus 3, Proposition \ref{proposition:a2Bounds} constrains $a_2$ to an interval of radius 3 once $a_1$ is known, whereas (\ref{equation:abounds}) would give a radius of 15.

Having bounded the interval as tightly as possible, we consider the search within.  We suppose we are seeking the value of a random variable $X$ with some distribution over $[M_0,M_1]$.  We assume that we start from an initial estimate $M$ and search outward in both directions using a standard baby-steps giant-steps search with all baby steps taken first (see \cite{Stein:BabyGiant} for a more general analysis).  Ignoring the boundaries, the cost of the search is
\begin{equation}
c = s + 2{\bf\vert} X-M {\bf\vert}/s
\end{equation}
group operations.  As our cost function is linear in $|X-M|$, we minimize the mean absolute error in our estimate by setting $M$ to the median value of $X$ and $s = \sqrt{2E}$, where $E$ is the expectation of $\vert X-M\vert$. This holds for any distribution on $X$, we simply need the median value of $X$ and its expected distance from it.

If we consider $p(z) = L_q(zq^{-1/2})$ as a ``random'' unitary symplectic polynomial, a natural distribution for $p(z)$ can be derived from the Haar measure on the compact Lie group $USp(2g)$ (the group of $2g\times 2g$ matrices over $\mathbb{C}$ that are both unitary and symplectic).  Each $p(z)$ corresponds to a conjugacy class of matrices with $p(z)$ as their characteristic polynomial.  Let the eigenvalues of a random matrix in $USp(2g)$ be $e^{\pm i\theta_1}$, \ldots, $e^{\pm i\theta_g}$, with $\theta_j \in [0,\pi)$.  The joint probability density function on the $\theta_j$ given by the Haar measure on $USp(2g)$ is
\begin{equation}\label{EigenvalueDistribution}
\mu(USp(2g)) = \frac{1}{g!}\left(\prod_{j<k}\left(2\cos\theta_j - 2\cos\theta_k\right)\right)^2\prod_j\frac{2}{\pi}\sin^2\theta_j d\theta_j.
\end{equation}
This distribution is derived from the Weyl integration formula \cite[p. 218]{Weyl:ClassicalGroups} and can be found in \cite[p. 107]{Katz:RandomMatrices}.  For $g=1$, this simplifies to $(2/\pi)\sin^2\theta d\theta$, which corresponds to the Sato-Tate distribution.  We may apply (\ref{EigenvalueDistribution}) to compute various statistical properties of random unitary symplectic polynomials.  The coefficient $a_1$ is simply the negative sum of the eigenvalues,
\begin{equation}
a_1 = - \sum_{j=1}^g 2\cos\theta_j,
\end{equation}
and we find that the median (and expectation) of $a_1$ is 0.  In genus 1, the expected distance of $a_1$ from its median is
\begin{equation}
{\bf E}\left[|a_1|\right] = \frac{2}{\pi}\int_0^\pi |2\cos\theta|\sin^2\theta d\theta = \frac{8}{3\pi}.
\end{equation}
The value $8/(3\pi) \approx 0.8488$ is not much smaller than $1$, which corresponds to a uniform distribution, so the potential benefit is small in genus 1.  In genus 2, however, the expected distance of $a_1$ from its median is $4096/(625\pi^2) \approx 0.7905$, versus an expected distance of 2 for the uniform distribution.  The corresponding values for genus 3 are $\approx 0.7985$ and 3.

Given the value of $a_1$ we can take this approach further, computing the median and expected distance for $a_2$ conditioned on $a_1$.  Applying (\ref{EigenvalueDistribution}), we precompute a table of median and expected distance values for $a_2$ for various ranges of $a_1$.  In genus 3, we find that the largest expected distance for $a_2$ given $a_1$ is about $0.66$, much smaller than the value 7.5 for a uniform distribution of $a_2$ over the interval given by (\ref{equation:abounds}).

Of course such optimizations are effective only when the polynomials $L_p(T)$ for a particular curve and relatively small values of $p$ actually correspond to (apparently) random unitary symplectic polynomials.  For $g>1$, it is not known whether this occurs at all, even as $p\to\infty$.\footnote{Results are known for certain universal families of curves, e.g. \cite[Thm. 10.8.2]{Katz:RandomMatrices}.}  In genus 1, while the Sato-Tate conjecture is now largely proven over $\mathbb{Q}$
\cite{Harris:SatoTateProof}, the convergence rate remains the subject of conjecture.  Indeed, the investigation of such questions was one motivation for undertaking these computations.  It is only natural to ask whether our assumptions are met.

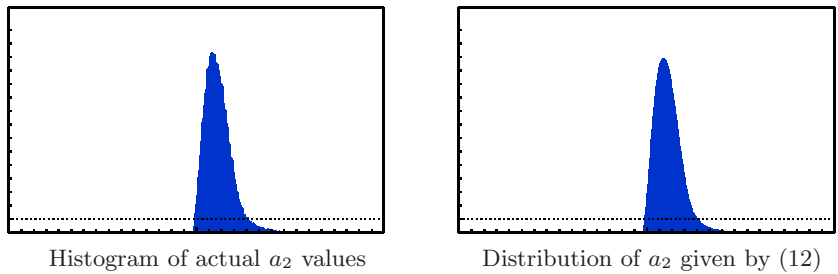
\vspace{12pt}
\begin{picture}(2200,600)
\color[rgb]{0.0,0.2,0.8}
\drawline(492,0)(492,2)(493,2)(493,0)
\drawline(493,0)(493,2)(494,2)(494,0)
\drawline(494,0)(494,7)(495,7)(495,0)
\drawline(495,0)(495,13)(496,13)(496,0)
\drawline(496,0)(496,17)(497,17)(497,0)
\drawline(497,0)(497,28)(498,28)(498,0)
\drawline(498,0)(498,40)(499,40)(499,0)
\drawline(499,0)(499,50)(500,50)(500,0)
\drawline(500,0)(500,58)(501,58)(501,0)
\drawline(501,0)(501,61)(502,61)(502,0)
\drawline(502,0)(502,74)(503,74)(503,0)
\drawline(503,0)(503,86)(504,86)(504,0)
\drawline(504,0)(504,100)(505,100)(505,0)
\drawline(505,0)(505,114)(506,114)(506,0)
\drawline(506,0)(506,138)(507,138)(507,0)
\drawline(507,0)(507,147)(508,147)(508,0)
\drawline(508,0)(508,158)(509,158)(509,0)
\drawline(509,0)(509,165)(510,165)(510,0)
\drawline(510,0)(510,191)(511,191)(511,0)
\drawline(511,0)(511,205)(512,205)(512,0)
\drawline(512,0)(512,195)(513,195)(513,0)
\drawline(513,0)(513,228)(514,228)(514,0)
\drawline(514,0)(514,244)(515,244)(515,0)
\drawline(515,0)(515,248)(516,248)(516,0)
\drawline(516,0)(516,259)(517,259)(517,0)
\drawline(517,0)(517,294)(518,294)(518,0)
\drawline(518,0)(518,274)(519,274)(519,0)
\drawline(519,0)(519,301)(520,301)(520,0)
\drawline(520,0)(520,301)(521,301)(521,0)
\drawline(521,0)(521,338)(522,338)(522,0)
\drawline(522,0)(522,311)(523,311)(523,0)
\drawline(523,0)(523,324)(524,324)(524,0)
\drawline(524,0)(524,368)(525,368)(525,0)
\drawline(525,0)(525,356)(526,356)(526,0)
\drawline(526,0)(526,366)(527,366)(527,0)
\drawline(527,0)(527,369)(528,369)(528,0)
\drawline(528,0)(528,381)(529,381)(529,0)
\drawline(529,0)(529,409)(530,409)(530,0)
\drawline(530,0)(530,402)(531,402)(531,0)
\drawline(531,0)(531,435)(532,435)(532,0)
\drawline(532,0)(532,397)(533,397)(533,0)
\drawline(533,0)(533,419)(534,419)(534,0)
\drawline(534,0)(534,439)(535,439)(535,0)
\drawline(535,0)(535,440)(536,440)(536,0)
\drawline(536,0)(536,424)(537,424)(537,0)
\drawline(537,0)(537,456)(538,456)(538,0)
\drawline(538,0)(538,449)(539,449)(539,0)
\drawline(539,0)(539,469)(540,469)(540,0)
\drawline(540,0)(540,454)(541,454)(541,0)
\drawline(541,0)(541,462)(542,462)(542,0)
\drawline(542,0)(542,478)(543,478)(543,0)
\drawline(543,0)(543,451)(544,451)(544,0)
\drawline(544,0)(544,459)(545,459)(545,0)
\drawline(545,0)(545,434)(546,434)(546,0)
\drawline(546,0)(546,468)(547,468)(547,0)
\drawline(547,0)(547,462)(548,462)(548,0)
\drawline(548,0)(548,461)(549,461)(549,0)
\drawline(549,0)(549,472)(550,472)(550,0)
\drawline(550,0)(550,467)(551,467)(551,0)
\drawline(551,0)(551,453)(552,453)(552,0)
\drawline(552,0)(552,443)(553,443)(553,0)
\drawline(553,0)(553,432)(554,432)(554,0)
\drawline(554,0)(554,449)(555,449)(555,0)
\drawline(555,0)(555,438)(556,438)(556,0)
\drawline(556,0)(556,449)(557,449)(557,0)
\drawline(557,0)(557,447)(558,447)(558,0)
\drawline(558,0)(558,433)(559,433)(559,0)
\drawline(559,0)(559,433)(560,433)(560,0)
\drawline(560,0)(560,425)(561,425)(561,0)
\drawline(561,0)(561,415)(562,415)(562,0)
\drawline(562,0)(562,417)(563,417)(563,0)
\drawline(563,0)(563,405)(564,405)(564,0)
\drawline(564,0)(564,377)(565,377)(565,0)
\drawline(565,0)(565,396)(566,396)(566,0)
\drawline(566,0)(566,395)(567,395)(567,0)
\drawline(567,0)(567,394)(568,394)(568,0)
\drawline(568,0)(568,381)(569,381)(569,0)
\drawline(569,0)(569,381)(570,381)(570,0)
\drawline(570,0)(570,381)(571,381)(571,0)
\drawline(571,0)(571,388)(572,388)(572,0)
\drawline(572,0)(572,338)(573,338)(573,0)
\drawline(573,0)(573,343)(574,343)(574,0)
\drawline(574,0)(574,328)(575,328)(575,0)
\drawline(575,0)(575,317)(576,317)(576,0)
\drawline(576,0)(576,306)(577,306)(577,0)
\drawline(577,0)(577,325)(578,325)(578,0)
\drawline(578,0)(578,298)(579,298)(579,0)
\drawline(579,0)(579,293)(580,293)(580,0)
\drawline(580,0)(580,291)(581,291)(581,0)
\drawline(581,0)(581,280)(582,280)(582,0)
\drawline(582,0)(582,274)(583,274)(583,0)
\drawline(583,0)(583,263)(584,263)(584,0)
\drawline(584,0)(584,279)(585,279)(585,0)
\drawline(585,0)(585,240)(586,240)(586,0)
\drawline(586,0)(586,234)(587,234)(587,0)
\drawline(587,0)(587,223)(588,223)(588,0)
\drawline(588,0)(588,234)(589,234)(589,0)
\drawline(589,0)(589,220)(590,220)(590,0)
\drawline(590,0)(590,231)(591,231)(591,0)
\drawline(591,0)(591,198)(592,198)(592,0)
\drawline(592,0)(592,190)(593,190)(593,0)
\drawline(593,0)(593,198)(594,198)(594,0)
\drawline(594,0)(594,190)(595,190)(595,0)
\drawline(595,0)(595,183)(596,183)(596,0)
\drawline(596,0)(596,193)(597,193)(597,0)
\drawline(597,0)(597,183)(598,183)(598,0)
\drawline(598,0)(598,158)(599,158)(599,0)
\drawline(599,0)(599,160)(600,160)(600,0)
\drawline(600,0)(600,149)(601,149)(601,0)
\drawline(601,0)(601,140)(602,140)(602,0)
\drawline(602,0)(602,131)(603,131)(603,0)
\drawline(603,0)(603,127)(604,127)(604,0)
\drawline(604,0)(604,128)(605,128)(605,0)
\drawline(605,0)(605,113)(606,113)(606,0)
\drawline(606,0)(606,105)(607,105)(607,0)
\drawline(607,0)(607,106)(608,106)(608,0)
\drawline(608,0)(608,105)(609,105)(609,0)
\drawline(609,0)(609,106)(610,106)(610,0)
\drawline(610,0)(610,102)(611,102)(611,0)
\drawline(611,0)(611,98)(612,98)(612,0)
\drawline(612,0)(612,83)(613,83)(613,0)
\drawline(613,0)(613,88)(614,88)(614,0)
\drawline(614,0)(614,78)(615,78)(615,0)
\drawline(615,0)(615,79)(616,79)(616,0)
\drawline(616,0)(616,77)(617,77)(617,0)
\drawline(617,0)(617,72)(618,72)(618,0)
\drawline(618,0)(618,70)(619,70)(619,0)
\drawline(619,0)(619,63)(620,63)(620,0)
\drawline(620,0)(620,58)(621,58)(621,0)
\drawline(621,0)(621,56)(622,56)(622,0)
\drawline(622,0)(622,65)(623,65)(623,0)
\drawline(623,0)(623,57)(624,57)(624,0)
\drawline(624,0)(624,54)(625,54)(625,0)
\drawline(625,0)(625,45)(626,45)(626,0)
\drawline(626,0)(626,41)(627,41)(627,0)
\drawline(627,0)(627,46)(628,46)(628,0)
\drawline(628,0)(628,40)(629,40)(629,0)
\drawline(629,0)(629,47)(630,47)(630,0)
\drawline(630,0)(630,43)(631,43)(631,0)
\drawline(631,0)(631,38)(632,38)(632,0)
\drawline(632,0)(632,33)(633,33)(633,0)
\drawline(633,0)(633,38)(634,38)(634,0)
\drawline(634,0)(634,37)(635,37)(635,0)
\drawline(635,0)(635,32)(636,32)(636,0)
\drawline(636,0)(636,32)(637,32)(637,0)
\drawline(637,0)(637,37)(638,37)(638,0)
\drawline(638,0)(638,28)(639,28)(639,0)
\drawline(639,0)(639,29)(640,29)(640,0)
\drawline(640,0)(640,30)(641,30)(641,0)
\drawline(641,0)(641,28)(642,28)(642,0)
\drawline(642,0)(642,20)(643,20)(643,0)
\drawline(643,0)(643,23)(644,23)(644,0)
\drawline(644,0)(644,21)(645,21)(645,0)
\drawline(645,0)(645,27)(646,27)(646,0)
\drawline(646,0)(646,22)(647,22)(647,0)
\drawline(647,0)(647,24)(648,24)(648,0)
\drawline(648,0)(648,19)(649,19)(649,0)
\drawline(649,0)(649,23)(650,23)(650,0)
\drawline(650,0)(650,22)(651,22)(651,0)
\drawline(651,0)(651,16)(652,16)(652,0)
\drawline(652,0)(652,20)(653,20)(653,0)
\drawline(653,0)(653,25)(654,25)(654,0)
\drawline(654,0)(654,20)(655,20)(655,0)
\drawline(655,0)(655,20)(656,20)(656,0)
\drawline(656,0)(656,13)(657,13)(657,0)
\drawline(657,0)(657,15)(658,15)(658,0)
\drawline(658,0)(658,19)(659,19)(659,0)
\drawline(659,0)(659,10)(660,10)(660,0)
\drawline(660,0)(660,12)(661,12)(661,0)
\drawline(661,0)(661,12)(662,12)(662,0)
\drawline(662,0)(662,14)(663,14)(663,0)
\drawline(663,0)(663,10)(664,10)(664,0)
\drawline(664,0)(664,15)(665,15)(665,0)
\drawline(665,0)(665,9)(666,9)(666,0)
\drawline(666,0)(666,13)(667,13)(667,0)
\drawline(667,0)(667,13)(668,13)(668,0)
\drawline(668,0)(668,13)(669,13)(669,0)
\drawline(669,0)(669,10)(670,10)(670,0)
\drawline(670,0)(670,7)(671,7)(671,0)
\drawline(671,0)(671,7)(672,7)(672,0)
\drawline(672,0)(672,10)(673,10)(673,0)
\drawline(673,0)(673,7)(674,7)(674,0)
\drawline(674,0)(674,10)(675,10)(675,0)
\drawline(675,0)(675,7)(676,7)(676,0)
\drawline(676,0)(676,7)(677,7)(677,0)
\drawline(677,0)(677,10)(678,10)(678,0)
\drawline(678,0)(678,6)(679,6)(679,0)
\drawline(679,0)(679,8)(680,8)(680,0)
\drawline(680,0)(680,8)(681,8)(681,0)
\drawline(681,0)(681,5)(682,5)(682,0)
\drawline(682,0)(682,7)(683,7)(683,0)
\drawline(683,0)(683,2)(684,2)(684,0)
\drawline(684,0)(684,5)(685,5)(685,0)
\drawline(685,0)(685,4)(686,4)(686,0)
\drawline(686,0)(686,4)(687,4)(687,0)
\drawline(687,0)(687,7)(688,7)(688,0)
\drawline(688,0)(688,3)(689,3)(689,0)
\drawline(689,0)(689,5)(690,5)(690,0)
\drawline(690,0)(690,6)(691,6)(691,0)
\drawline(691,0)(691,4)(692,4)(692,0)
\drawline(692,0)(692,8)(693,8)(693,0)
\drawline(693,0)(693,7)(694,7)(694,0)
\drawline(694,0)(694,6)(695,6)(695,0)
\drawline(695,0)(695,5)(696,5)(696,0)
\drawline(696,0)(696,3)(697,3)(697,0)
\drawline(697,0)(697,2)(698,2)(698,0)
\drawline(698,0)(698,4)(699,4)(699,0)
\drawline(699,0)(699,5)(700,5)(700,0)
\drawline(700,0)(700,5)(701,5)(701,0)
\drawline(701,0)(701,3)(702,3)(702,0)
\drawline(702,0)(702,3)(703,3)(703,0)
\drawline(703,0)(703,3)(704,3)(704,0)
\drawline(704,0)(704,3)(705,3)(705,0)
\drawline(705,0)(705,4)(706,4)(706,0)
\drawline(706,0)(706,3)(707,3)(707,0)
\drawline(707,0)(707,4)(708,4)(708,0)
\drawline(708,0)(708,2)(709,2)(709,0)
\drawline(709,0)(709,3)(710,3)(710,0)
\drawline(710,0)(710,2)(711,2)(711,0)
\drawline(711,0)(711,3)(712,3)(712,0)
\drawline(712,0)(712,1)(713,1)(713,0)
\drawline(713,0)(713,4)(714,4)(714,0)
\drawline(714,0)(714,4)(715,4)(715,0)
\drawline(715,0)(715,2)(716,2)(716,0)
\drawline(716,0)(716,2)(717,2)(717,0)
\drawline(717,0)(717,1)(718,1)(718,0)
\drawline(719,0)(719,1)(720,1)(720,0)
\drawline(720,0)(720,1)(721,1)(721,0)
\drawline(721,0)(721,3)(722,3)(722,0)
\drawline(722,0)(722,2)(723,2)(723,0)
\drawline(726,0)(726,2)(727,2)(727,0)
\drawline(727,0)(727,1)(728,1)(728,0)
\drawline(728,0)(728,1)(729,1)(729,0)
\drawline(729,0)(729,2)(730,2)(730,0)
\drawline(731,0)(731,2)(732,2)(732,0)
\drawline(732,0)(732,2)(733,2)(733,0)
\drawline(734,0)(734,1)(735,1)(735,0)
\drawline(735,0)(735,1)(736,1)(736,0)
\drawline(736,0)(736,1)(737,1)(737,0)
\drawline(738,0)(738,1)(739,1)(739,0)
\drawline(739,0)(739,1)(740,1)(740,0)
\drawline(740,0)(740,1)(741,1)(741,0)
\drawline(746,0)(746,1)(747,1)(747,0)
\drawline(747,0)(747,1)(748,1)(748,0)
\drawline(750,0)(750,1)(751,1)(751,0)
\drawline(756,0)(756,1)(757,1)(757,0)
\drawline(759,0)(759,1)(760,1)(760,0)
\drawline(764,0)(764,1)(765,1)(765,0)
\drawline(774,0)(774,1)(775,1)(775,0)

\drawline(1691,0)(1691,1)(1692,1)(1692,0)
\drawline(1692,0)(1692,3)(1693,3)(1693,0)
\drawline(1693,0)(1693,7)(1694,7)(1694,0)
\drawline(1694,0)(1694,13)(1695,13)(1695,0)
\drawline(1695,0)(1695,20)(1696,20)(1696,0)
\drawline(1696,0)(1696,27)(1697,27)(1697,0)
\drawline(1697,0)(1697,36)(1698,36)(1698,0)
\drawline(1698,0)(1698,46)(1699,46)(1699,0)
\drawline(1699,0)(1699,56)(1700,56)(1700,0)
\drawline(1700,0)(1700,67)(1701,67)(1701,0)
\drawline(1701,0)(1701,79)(1702,79)(1702,0)
\drawline(1702,0)(1702,91)(1703,91)(1703,0)
\drawline(1703,0)(1703,103)(1704,103)(1704,0)
\drawline(1704,0)(1704,116)(1705,116)(1705,0)
\drawline(1705,0)(1705,129)(1706,129)(1706,0)
\drawline(1706,0)(1706,142)(1707,142)(1707,0)
\drawline(1707,0)(1707,156)(1708,156)(1708,0)
\drawline(1708,0)(1708,169)(1709,169)(1709,0)
\drawline(1709,0)(1709,183)(1710,183)(1710,0)
\drawline(1710,0)(1710,196)(1711,196)(1711,0)
\drawline(1711,0)(1711,209)(1712,209)(1712,0)
\drawline(1712,0)(1712,223)(1713,223)(1713,0)
\drawline(1713,0)(1713,236)(1714,236)(1714,0)
\drawline(1714,0)(1714,249)(1715,249)(1715,0)
\drawline(1715,0)(1715,261)(1716,261)(1716,0)
\drawline(1716,0)(1716,274)(1717,274)(1717,0)
\drawline(1717,0)(1717,286)(1718,286)(1718,0)
\drawline(1718,0)(1718,298)(1719,298)(1719,0)
\drawline(1719,0)(1719,309)(1720,309)(1720,0)
\drawline(1720,0)(1720,321)(1721,321)(1721,0)
\drawline(1721,0)(1721,332)(1722,332)(1722,0)
\drawline(1722,0)(1722,342)(1723,342)(1723,0)
\drawline(1723,0)(1723,352)(1724,352)(1724,0)
\drawline(1724,0)(1724,362)(1725,362)(1725,0)
\drawline(1725,0)(1725,371)(1726,371)(1726,0)
\drawline(1726,0)(1726,380)(1727,380)(1727,0)
\drawline(1727,0)(1727,388)(1728,388)(1728,0)
\drawline(1728,0)(1728,396)(1729,396)(1729,0)
\drawline(1729,0)(1729,404)(1730,404)(1730,0)
\drawline(1730,0)(1730,411)(1731,411)(1731,0)
\drawline(1731,0)(1731,418)(1732,418)(1732,0)
\drawline(1732,0)(1732,424)(1733,424)(1733,0)
\drawline(1733,0)(1733,429)(1734,429)(1734,0)
\drawline(1734,0)(1734,435)(1735,435)(1735,0)
\drawline(1735,0)(1735,439)(1736,439)(1736,0)
\drawline(1736,0)(1736,444)(1737,444)(1737,0)
\drawline(1737,0)(1737,447)(1738,447)(1738,0)
\drawline(1738,0)(1738,451)(1739,451)(1739,0)
\drawline(1739,0)(1739,454)(1740,454)(1740,0)
\drawline(1740,0)(1740,456)(1741,456)(1741,0)
\drawline(1741,0)(1741,458)(1742,458)(1742,0)
\drawline(1742,0)(1742,459)(1743,459)(1743,0)
\drawline(1743,0)(1743,461)(1744,461)(1744,0)
\drawline(1744,0)(1744,461)(1745,461)(1745,0)
\drawline(1745,0)(1745,461)(1746,461)(1746,0)
\drawline(1746,0)(1746,461)(1747,461)(1747,0)
\drawline(1747,0)(1747,461)(1748,461)(1748,0)
\drawline(1748,0)(1748,460)(1749,460)(1749,0)
\drawline(1749,0)(1749,458)(1750,458)(1750,0)
\drawline(1750,0)(1750,456)(1751,456)(1751,0)
\drawline(1751,0)(1751,454)(1752,454)(1752,0)
\drawline(1752,0)(1752,452)(1753,452)(1753,0)
\drawline(1753,0)(1753,449)(1754,449)(1754,0)
\drawline(1754,0)(1754,446)(1755,446)(1755,0)
\drawline(1755,0)(1755,442)(1756,442)(1756,0)
\drawline(1756,0)(1756,438)(1757,438)(1757,0)
\drawline(1757,0)(1757,434)(1758,434)(1758,0)
\drawline(1758,0)(1758,429)(1759,429)(1759,0)
\drawline(1759,0)(1759,425)(1760,425)(1760,0)
\drawline(1760,0)(1760,420)(1761,420)(1761,0)
\drawline(1761,0)(1761,414)(1762,414)(1762,0)
\drawline(1762,0)(1762,409)(1763,409)(1763,0)
\drawline(1763,0)(1763,403)(1764,403)(1764,0)
\drawline(1764,0)(1764,397)(1765,397)(1765,0)
\drawline(1765,0)(1765,391)(1766,391)(1766,0)
\drawline(1766,0)(1766,385)(1767,385)(1767,0)
\drawline(1767,0)(1767,378)(1768,378)(1768,0)
\drawline(1768,0)(1768,372)(1769,372)(1769,0)
\drawline(1769,0)(1769,365)(1770,365)(1770,0)
\drawline(1770,0)(1770,358)(1771,358)(1771,0)
\drawline(1771,0)(1771,351)(1772,351)(1772,0)
\drawline(1772,0)(1772,343)(1773,343)(1773,0)
\drawline(1773,0)(1773,336)(1774,336)(1774,0)
\drawline(1774,0)(1774,329)(1775,329)(1775,0)
\drawline(1775,0)(1775,321)(1776,321)(1776,0)
\drawline(1776,0)(1776,314)(1777,314)(1777,0)
\drawline(1777,0)(1777,306)(1778,306)(1778,0)
\drawline(1778,0)(1778,298)(1779,298)(1779,0)
\drawline(1779,0)(1779,291)(1780,291)(1780,0)
\drawline(1780,0)(1780,283)(1781,283)(1781,0)
\drawline(1781,0)(1781,275)(1782,275)(1782,0)
\drawline(1782,0)(1782,268)(1783,268)(1783,0)
\drawline(1783,0)(1783,260)(1784,260)(1784,0)
\drawline(1784,0)(1784,252)(1785,252)(1785,0)
\drawline(1785,0)(1785,245)(1786,245)(1786,0)
\drawline(1786,0)(1786,237)(1787,237)(1787,0)
\drawline(1787,0)(1787,230)(1788,230)(1788,0)
\drawline(1788,0)(1788,222)(1789,222)(1789,0)
\drawline(1789,0)(1789,215)(1790,215)(1790,0)
\drawline(1790,0)(1790,208)(1791,208)(1791,0)
\drawline(1791,0)(1791,201)(1792,201)(1792,0)
\drawline(1792,0)(1792,194)(1793,194)(1793,0)
\drawline(1793,0)(1793,187)(1794,187)(1794,0)
\drawline(1794,0)(1794,180)(1795,180)(1795,0)
\drawline(1795,0)(1795,173)(1796,173)(1796,0)
\drawline(1796,0)(1796,167)(1797,167)(1797,0)
\drawline(1797,0)(1797,160)(1798,160)(1798,0)
\drawline(1798,0)(1798,154)(1799,154)(1799,0)
\drawline(1799,0)(1799,148)(1800,148)(1800,0)
\drawline(1800,0)(1800,142)(1801,142)(1801,0)
\drawline(1801,0)(1801,137)(1802,137)(1802,0)
\drawline(1802,0)(1802,131)(1803,131)(1803,0)
\drawline(1803,0)(1803,125)(1804,125)(1804,0)
\drawline(1804,0)(1804,120)(1805,120)(1805,0)
\drawline(1805,0)(1805,115)(1806,115)(1806,0)
\drawline(1806,0)(1806,110)(1807,110)(1807,0)
\drawline(1807,0)(1807,106)(1808,106)(1808,0)
\drawline(1808,0)(1808,101)(1809,101)(1809,0)
\drawline(1809,0)(1809,97)(1810,97)(1810,0)
\drawline(1810,0)(1810,92)(1811,92)(1811,0)
\drawline(1811,0)(1811,88)(1812,88)(1812,0)
\drawline(1812,0)(1812,85)(1813,85)(1813,0)
\drawline(1813,0)(1813,81)(1814,81)(1814,0)
\drawline(1814,0)(1814,78)(1815,78)(1815,0)
\drawline(1815,0)(1815,74)(1816,74)(1816,0)
\drawline(1816,0)(1816,71)(1817,71)(1817,0)
\drawline(1817,0)(1817,68)(1818,68)(1818,0)
\drawline(1818,0)(1818,65)(1819,65)(1819,0)
\drawline(1819,0)(1819,63)(1820,63)(1820,0)
\drawline(1820,0)(1820,60)(1821,60)(1821,0)
\drawline(1821,0)(1821,58)(1822,58)(1822,0)
\drawline(1822,0)(1822,55)(1823,55)(1823,0)
\drawline(1823,0)(1823,53)(1824,53)(1824,0)
\drawline(1824,0)(1824,51)(1825,51)(1825,0)
\drawline(1825,0)(1825,49)(1826,49)(1826,0)
\drawline(1826,0)(1826,47)(1827,47)(1827,0)
\drawline(1827,0)(1827,45)(1828,45)(1828,0)
\drawline(1828,0)(1828,44)(1829,44)(1829,0)
\drawline(1829,0)(1829,42)(1830,42)(1830,0)
\drawline(1830,0)(1830,40)(1831,40)(1831,0)
\drawline(1831,0)(1831,39)(1832,39)(1832,0)
\drawline(1832,0)(1832,37)(1833,37)(1833,0)
\drawline(1833,0)(1833,36)(1834,36)(1834,0)
\drawline(1834,0)(1834,35)(1835,35)(1835,0)
\drawline(1835,0)(1835,33)(1836,33)(1836,0)
\drawline(1836,0)(1836,32)(1837,32)(1837,0)
\drawline(1837,0)(1837,31)(1838,31)(1838,0)
\drawline(1838,0)(1838,30)(1839,30)(1839,0)
\drawline(1839,0)(1839,29)(1840,29)(1840,0)
\drawline(1840,0)(1840,28)(1841,28)(1841,0)
\drawline(1841,0)(1841,27)(1842,27)(1842,0)
\drawline(1842,0)(1842,26)(1843,26)(1843,0)
\drawline(1843,0)(1843,25)(1844,25)(1844,0)
\drawline(1844,0)(1844,24)(1845,24)(1845,0)
\drawline(1845,0)(1845,23)(1846,23)(1846,0)
\drawline(1846,0)(1846,22)(1847,22)(1847,0)
\drawline(1847,0)(1847,21)(1848,21)(1848,0)
\drawline(1848,0)(1848,21)(1849,21)(1849,0)
\drawline(1849,0)(1849,20)(1850,20)(1850,0)
\drawline(1850,0)(1850,19)(1851,19)(1851,0)
\drawline(1851,0)(1851,18)(1852,18)(1852,0)
\drawline(1852,0)(1852,18)(1853,18)(1853,0)
\drawline(1853,0)(1853,17)(1854,17)(1854,0)
\drawline(1854,0)(1854,17)(1855,17)(1855,0)
\drawline(1855,0)(1855,16)(1856,16)(1856,0)
\drawline(1856,0)(1856,15)(1857,15)(1857,0)
\drawline(1857,0)(1857,15)(1858,15)(1858,0)
\drawline(1858,0)(1858,14)(1859,14)(1859,0)
\drawline(1859,0)(1859,14)(1860,14)(1860,0)
\drawline(1860,0)(1860,13)(1861,13)(1861,0)
\drawline(1861,0)(1861,13)(1862,13)(1862,0)
\drawline(1862,0)(1862,12)(1863,12)(1863,0)
\drawline(1863,0)(1863,12)(1864,12)(1864,0)
\drawline(1864,0)(1864,11)(1865,11)(1865,0)
\drawline(1865,0)(1865,11)(1866,11)(1866,0)
\drawline(1866,0)(1866,11)(1867,11)(1867,0)
\drawline(1867,0)(1867,10)(1868,10)(1868,0)
\drawline(1868,0)(1868,10)(1869,10)(1869,0)
\drawline(1869,0)(1869,10)(1870,10)(1870,0)
\drawline(1870,0)(1870,9)(1871,9)(1871,0)
\drawline(1871,0)(1871,9)(1872,9)(1872,0)
\drawline(1872,0)(1872,9)(1873,9)(1873,0)
\drawline(1873,0)(1873,8)(1874,8)(1874,0)
\drawline(1874,0)(1874,8)(1875,8)(1875,0)
\drawline(1875,0)(1875,8)(1876,8)(1876,0)
\drawline(1876,0)(1876,7)(1877,7)(1877,0)
\drawline(1877,0)(1877,7)(1878,7)(1878,0)
\drawline(1878,0)(1878,7)(1879,7)(1879,0)
\drawline(1879,0)(1879,7)(1880,7)(1880,0)
\drawline(1880,0)(1880,6)(1881,6)(1881,0)
\drawline(1881,0)(1881,6)(1882,6)(1882,0)
\drawline(1882,0)(1882,6)(1883,6)(1883,0)
\drawline(1883,0)(1883,6)(1884,6)(1884,0)
\drawline(1884,0)(1884,5)(1885,5)(1885,0)
\drawline(1885,0)(1885,5)(1886,5)(1886,0)
\drawline(1886,0)(1886,5)(1887,5)(1887,0)
\drawline(1887,0)(1887,5)(1888,5)(1888,0)
\drawline(1888,0)(1888,5)(1889,5)(1889,0)
\drawline(1889,0)(1889,5)(1890,5)(1890,0)
\drawline(1890,0)(1890,4)(1891,4)(1891,0)
\drawline(1891,0)(1891,4)(1892,4)(1892,0)
\drawline(1892,0)(1892,4)(1893,4)(1893,0)
\drawline(1893,0)(1893,4)(1894,4)(1894,0)
\drawline(1894,0)(1894,4)(1895,4)(1895,0)
\drawline(1895,0)(1895,4)(1896,4)(1896,0)
\drawline(1896,0)(1896,4)(1897,4)(1897,0)
\drawline(1897,0)(1897,3)(1898,3)(1898,0)
\drawline(1898,0)(1898,3)(1899,3)(1899,0)
\drawline(1899,0)(1899,3)(1900,3)(1900,0)
\drawline(1900,0)(1900,3)(1901,3)(1901,0)
\drawline(1901,0)(1901,3)(1902,3)(1902,0)
\drawline(1902,0)(1902,3)(1903,3)(1903,0)
\drawline(1903,0)(1903,3)(1904,3)(1904,0)
\drawline(1904,0)(1904,3)(1905,3)(1905,0)
\drawline(1905,0)(1905,3)(1906,3)(1906,0)
\drawline(1906,0)(1906,2)(1907,2)(1907,0)
\drawline(1907,0)(1907,2)(1908,2)(1908,0)
\drawline(1908,0)(1908,2)(1909,2)(1909,0)
\drawline(1909,0)(1909,2)(1910,2)(1910,0)
\drawline(1910,0)(1910,2)(1911,2)(1911,0)
\drawline(1911,0)(1911,2)(1912,2)(1912,0)
\drawline(1912,0)(1912,2)(1913,2)(1913,0)
\drawline(1913,0)(1913,2)(1914,2)(1914,0)
\drawline(1914,0)(1914,2)(1915,2)(1915,0)
\drawline(1915,0)(1915,2)(1916,2)(1916,0)
\drawline(1916,0)(1916,2)(1917,2)(1917,0)
\drawline(1917,0)(1917,2)(1918,2)(1918,0)
\drawline(1918,0)(1918,2)(1919,2)(1919,0)
\drawline(1919,0)(1919,1)(1920,1)(1920,0)
\drawline(1920,0)(1920,1)(1921,1)(1921,0)
\drawline(1921,0)(1921,1)(1922,1)(1922,0)
\drawline(1922,0)(1922,1)(1923,1)(1923,0)
\drawline(1923,0)(1923,1)(1924,1)(1924,0)
\drawline(1924,0)(1924,1)(1925,1)(1925,0)
\drawline(1925,0)(1925,1)(1926,1)(1926,0)
\drawline(1926,0)(1926,1)(1927,1)(1927,0)
\drawline(1927,0)(1927,1)(1928,1)(1928,0)
\drawline(1928,0)(1928,1)(1929,1)(1929,0)
\drawline(1929,0)(1929,1)(1930,1)(1930,0)
\drawline(1930,0)(1930,1)(1931,1)(1931,0)
\drawline(1931,0)(1931,1)(1932,1)(1932,0)
\drawline(1932,0)(1932,1)(1933,1)(1933,0)
\drawline(1933,0)(1933,1)(1934,1)(1934,0)
\drawline(1934,0)(1934,1)(1935,1)(1935,0)
\drawline(1935,0)(1935,1)(1936,1)(1936,0)
\drawline(1936,0)(1936,1)(1937,1)(1937,0)
\drawline(1937,0)(1937,1)(1938,1)(1938,0)
\drawline(1938,0)(1938,1)(1939,1)(1939,0)
\drawline(1939,0)(1939,1)(1940,1)(1940,0)
\drawline(1940,0)(1940,1)(1941,1)(1941,0)
\drawline(1941,0)(1941,1)(1942,1)(1942,0)
\drawline(1942,0)(1942,1)(1943,1)(1943,0)
\drawline(1943,0)(1943,1)(1944,1)(1944,0)
\drawline(1944,0)(1944,1)(1945,1)(1945,0)
\drawline(1945,0)(1945,1)(1946,1)(1946,0)

\color{black}
\drawline(0,0)(0,600)(1000,600)(1000,0)(0,0)
\drawline(34,0)(34,8)
\drawline(68,0)(68,8)
\drawline(101,0)(101,8)
\drawline(134,0)(134,8)
\drawline(168,0)(168,8)
\drawline(201,0)(201,8)
\drawline(234,0)(234,8)
\drawline(268,0)(268,8)
\drawline(301,0)(301,8)
\drawline(334,0)(334,8)
\drawline(368,0)(368,8)
\drawline(401,0)(401,8)
\drawline(434,0)(434,8)
\drawline(468,0)(468,8)
\drawline(501,0)(501,12)
\drawline(534,0)(534,8)
\drawline(568,0)(568,8)
\drawline(601,0)(601,8)
\drawline(634,0)(634,8)
\drawline(668,0)(668,8)
\drawline(701,0)(701,8)
\drawline(734,0)(734,8)
\drawline(768,0)(768,8)
\drawline(801,0)(801,8)
\drawline(834,0)(834,8)
\drawline(868,0)(868,8)
\drawline(901,0)(901,8)
\drawline(934,0)(934,8)
\drawline(968,0)(968,8)
\drawline(0,36)(8,36)
\drawline(0,72)(8,72)
\drawline(0,108)(8,108)
\drawline(0,144)(8,144)
\drawline(0,180)(8,180)
\drawline(0,216)(8,216)
\drawline(0,252)(8,252)
\drawline(0,288)(8,288)
\drawline(0,324)(8,324)
\drawline(0,360)(8,360)
\drawline(0,396)(8,396)
\drawline(0,432)(8,432)
\drawline(0,468)(8,468)
\drawline(0,504)(8,504)
\drawline(0,540)(8,540)
\dottedline{12}(0,36)(999,36)
\color{black}
\drawline(1200,0)(1200,600)(2200,600)(2200,0)(1200,0)
\drawline(1234,0)(1234,8)
\drawline(1268,0)(1268,8)
\drawline(1301,0)(1301,8)
\drawline(1334,0)(1334,8)
\drawline(1368,0)(1368,8)
\drawline(1401,0)(1401,8)
\drawline(1434,0)(1434,8)
\drawline(1468,0)(1468,8)
\drawline(1501,0)(1501,8)
\drawline(1534,0)(1534,8)
\drawline(1568,0)(1568,8)
\drawline(1601,0)(1601,8)
\drawline(1634,0)(1634,8)
\drawline(1668,0)(1668,8)
\drawline(1701,0)(1701,12)
\drawline(1734,0)(1734,8)
\drawline(1768,0)(1768,8)
\drawline(1801,0)(1801,8)
\drawline(1834,0)(1834,8)
\drawline(1868,0)(1868,8)
\drawline(1901,0)(1901,8)
\drawline(1934,0)(1934,8)
\drawline(1968,0)(1968,8)
\drawline(2001,0)(2001,8)
\drawline(2034,0)(2034,8)
\drawline(2068,0)(2068,8)
\drawline(2101,0)(2101,8)
\drawline(2134,0)(2134,8)
\drawline(2168,0)(2168,8)
\drawline(1200,36)(1208,36)
\drawline(1200,72)(1208,72)
\drawline(1200,108)(1208,108)
\drawline(1200,144)(1208,144)
\drawline(1200,180)(1208,180)
\drawline(1200,216)(1208,216)
\drawline(1200,252)(1208,252)
\drawline(1200,288)(1208,288)
\drawline(1200,324)(1208,324)
\drawline(1200,360)(1208,360)
\drawline(1200,396)(1208,396)
\drawline(1200,432)(1208,432)
\drawline(1200,468)(1208,468)
\drawline(1200,504)(1208,504)
\drawline(1200,540)(1208,540)
\dottedline{12}(1200,36)(2199,36)
\end{picture}
\vspace{-6pt}
\begin{center}
\small
\hspace{6pt}Histogram of actual $a_2$ values\hspace{44pt}Distribution of $a_2$ given by (\ref{EigenvalueDistribution})
\normalsize
\end{center}
The figure on the left is a histogram of $a_2$ coefficient values obtained by computing $L_p(T)$ for $p\le 2^{24}$ for an arbitrarily chosen genus 3 curve (see Table \ref{table:FullTests}).  The figure on the right is the distribution of $a_2$ predicted by the Haar measure on $USp(2g)$, obtained by numerically integrating
\begin{equation}
a_2 = g + \prod_{j<k}4\cos\theta_j\cos\theta_k
\end{equation}
over the distribution in (\ref{EigenvalueDistribution}).  The dotted lines show the height of the uniform distribution.  Similarly matching graphs are found for the other coefficients.

This remarkable degree of convergence is typical for a randomly chosen curve.  We should note, however, that the generalized form of the Sato-Tate conjecture considered here applies only to curves whose Jacobian over $\mathbb{Q}$ has a trivial endomorphism ring (isomorphic to $\mathbb{Z}$), so there are exceptional cases.  In genus 1 these are curves with complex multiplication.  In higher genera, other exceptional cases occur, such as the genus 2 QM-curves considered in \cite{Hashimoto:SatoTateGenus2}.

\section{Results}

To compare different methods for computing $L_p(T)$ and to assess the feasible range of $L$-series computations, we conducted extensive performance tests.  Our test platform consisted of eight networked PCs, each equipped with a 2.5GHz AMD Athlon processor running a 64-bit Linux operating system. The point-counting and generic group algorithms were implemented using the techniques described in this paper, and we incorporated David Harvey's source code for the $p$-adic computations (the algorithm of \cite{Harvey:LargeCharacteristicKedlaya}, including recent improvements described in \cite{Harvey:KroneckerSubstitution}).  All code was compiled with the GNU C/C++ compiler using the options ``-O2 -m64 -mtune=k8" \cite{GNU}.

In genus 1 there are several existing implementations of the computation contemplated here: given an elliptic curve defined over $\mathbb{Q}$, determine the coefficient $a_1$ of $L_p(T) = pT^2 + a_1T + 1$ for all $p\le N$.  
We were able to compare our implementation with two software packages specifically optimized for this purpose: Magma \cite{Magma}, and the PARI library \cite{PARI} as incorporated in SAGE \cite{SAGE}.  The range of $N$ we could use in this comparison was necessarily limited; results for larger $N$ may be found in Table \ref{table:Genus1FullTests}.

\begin{table}
\begin{center}
\begin{tabular}{@{}rrrr@{}}
$N$ &\hspace{18pt} PARI & \hspace{18pt} Magma & \hspace{18pt} smalljac\\
\midrule
$2^{16}$&     0.26 &     0.29 &    0.07\\
$2^{17}$&     0.55 &     0.59 &    0.15\\
$2^{18}$&     1.17 &     1.24 &    0.30\\
$2^{19}$&     2.51 &     2.53 &    0.62\\
$2^{20}$&     5.46 &     5.26 &    1.29\\
$2^{21}$&    11.67 &    11.09 &    2.65\\
$2^{22}$&    25.46 &    23.31 &    5.53\\
$2^{23}$&    55.50 &    49.22 &   11.56\\
$2^{24}$&   123.02 &   104.50 &   24.31\\
$2^{25}$&   266.40 &   222.56 &   51.60\\
$2^{26}$&   598.16 &   476.74 &  110.29\\
$2^{27}$&  1367.46 &  1017.55 &  233.94\\
$2^{28}$&  3152.91 &  2159.87 &  498.46\\
$2^{29}$&  7317.01 &  4646.24 & 1065.28\\
$2^{30}$& 17167.29 & 10141.28 & 2292.74\\
\bottomrule
\vspace{2pt}
\end{tabular}
\caption{$L$-series computations in genus 1 (CPU seconds)}\label{table:Genus1Comparison}
\end{center}
\vspace{-12pt}
\begin{minipage}{1.0\linewidth}
\small
Each row lists CPU times for a single thread of execution to compute the coefficient $a_1$ of $L_p(T)$ for all $p\le N$, using the elliptic curve $y^2 = x^3 + 314159x +271828$.
In SAGE, the function aplist($N$) performs this computation via the PARI function \emph{ellap}(N).  The corresponding function in Magma is \emph{TracesOfFrobenius}($N$).  The column labeled ``smalljac'' list times for our implementation.
\normalsize
\end{minipage}
\end{table}

Before undertaking similar computations in genus 2 and 3, we first determined the appropriate algorithm to use for various ranges of $p$ using Table \ref{table:AlgorithmComparison}.  Each row gives timings for the algorithms considered here, averaged over a small sample of primes of similar size.

\begin{table}
\begin{center}
\begin{tabular}{@{}lcrrrcrrcr@{}}
& &\multicolumn{3}{c}{Genus 2 -- $L_p(T)$}& &\multicolumn{2}{c}{Genus 3 -- $L_p(T)$}& &\multicolumn{1}{c}{Genus 3 -- $a_1$}\\
\cmidrule(r){3-5}\cmidrule(r){7-8}\cmidrule(r){10-10}
\multicolumn{1}{c}{$p\approx 2^k$}&$\qquad$&
\multicolumn{1}{r}{pts/grp} &\multicolumn{1}{r}{$\quad$group}&\multicolumn{1}{r}{$\quad p$-adic} &$\qquad$&
\multicolumn{1}{r}{pts/grp}&\multicolumn{1}{r}{$\quad p$-adic/grp} &$\qquad$&\multicolumn{1}{r}{points}\\
\midrule
$2^{14}$& &  {\bf 0.22} & 0.55 &  4 & & {\bf 10} & 15 & & {\bf 0.12}\\
$2^{15}$& &  {\bf 0.34} & 0.88 &  6 & & {\bf 21} & 23 & & {\bf 0.23}\\
$2^{16}$& &  {\bf 0.56} & 1.33 &  8 & & 43 & {\bf 31} & & {\bf 0.45}\\
$2^{17}$& &  {\bf 0.98} & 2.21 & 11 & & 82 & {\bf 40} & & {\bf 0.89}\\
$2^{18}$& &  {\bf 1.82} & 3.42 & 17 & &  & {\bf 51} && {\bf 1.78}\\
$2^{19}$& &  {\bf 3.44} & 5.87 & 27  & &  & {\bf 67} && {\bf 3.57}\\
$2^{20}$& &  {\bf 7.98} & 10.1 & 40  & &  & {\bf 97} && {\bf 8.48}\\
$2^{21}$& & 18.9 & {\bf 17.9} &  66  & &       & {\bf 148} && {\bf 19.7}\\
$2^{22}$& & 52 & {\bf 35} & 104 & &       & {\bf 212} && {\bf 56}\\
$2^{23}$& &       & {\bf 54}   & 176 & &       & {\bf 355} && {\bf 123}\\
$2^{24}$& &       & {\bf 104}  & 288 & &       & {\bf 577} && 738\\
$2^{25}$& &       & {\bf 173}  & 494 & &       & {\bf 995} && 1870\\
$2^{26}$& &       & {\bf 306}  & 871 & &       & {\bf 1753} && 4550\\
$2^{27}$& &       & {\bf 505}  & 1532 & &       & {\bf 3070}  && 9800\\
\bottomrule
\vspace{2pt}
\end{tabular}
\caption{$L_p(T)$ computations (CPU milliseconds)}\label{table:AlgorithmComparison}
\end{center}
\vspace{-12pt}
\begin{minipage}{1.0\linewidth}
\small
Random curves of the appropriate genus were generated with coefficients uniformly distributed over $[1,2^k)$.  The polynomial $L_p(T)$ was then computed for 100 primes $\approx 2^k$, with the average CPU time listed.  Columns labeled ``pts/grp" compute $a_1$ by point counting over $\Fp$, followed by a group computation to obtain $L_p(T)$.  The column ``p-adic/grp" computes $L_p(T)$ mod $p$, then applies a group computation to get $L_p(T)$.  The rightmost column computes just the coefficient $a_1$, via point counting over $\Fp$.
\normalsize
\end{minipage}
\end{table}
\vspace{-12pt}

The task of computing $L$-series coefficients is well-suited to parallel computation.  We implemented a simple distributed program which partitions the range $[1,N]$ into subintervals $I_1$, $I_2$, \ldots, $I_m$, distributes the task of computing $L_p(T)$ for $p\in I_m$ to $n$ CPUs on a network, then collects and collates the results.  This is useful even on a single computer whose microprocessor may have two or more cores.  On our 8 node test platform we had 16 CPUs available for computation.  Tables \ref{table:Genus1FullTests} and \ref{table:FullTests} lists elapsed times for $L$-series computations in single and 8-node configurations.

For practical reasons, we limited the duration of any single test.  Larger computations could be undertaken with additional time and/or computing resources, without requiring software modifications.  As they stand, the results extend to values of $N$ substantially larger than any we could find in the literature.  

Source code for the software can be freely obtained under a GNU General Public License (GPL) and is expected to be incorporated into SAGE.  It is a pleasure to thank William Stein for access to the SAGE computational resources at the University of Washington, and especially David Harvey for providing the code used for the $p$-adic computations.
\begin{table}
\begin{center}
\begin{tabular}{@{}rcrrcrcrr@{}}
&&\multicolumn{2}{c}{Genus 1}&&&&\multicolumn{2}{c}{Genus 1}\\
\cmidrule(r){3-4}\cmidrule(r){8-9}
$N$&$\quad$&$\qquad\times 1$ & $\quad\qquad\times 8$ &$\qquad\qquad$&$N$& $\quad$& $\qquad\times 1$& $\quad\qquad\times 8$\\
\midrule
$2^{21}$&&  1.5  &  0.5 && $2^{30}$&&    20:43 &     2:41\\
$2^{22}$&&  3.1  &  0.7 && $2^{31}$&&    45:13 &     5:52\\ 
$2^{23}$&&  6.3  &  1.1 && $2^{32}$&&  1:45:45 &    13:12\\
$2^{24}$&& 13.3  &  2.0 && $2^{33}$&&  4:24:50 &    32:51\\
$2^{25}$&& 28.2  &  4.2 && $2^{34}$&& 10:16:11 &  1:16:18\\
$2^{26}$&& 59.2  &  8.1 && $2^{35}$&& 23:15:58 &  2:52:47\\
$2^{27}$&& 126.2 & 16.6 && $2^{36}$&&          &  6:29:46\\
$2^{28}$&& 271.3 & 35.1 && $2^{37}$&&          & 14:44:33\\
$2^{29}$&& 578.0 & 74.5 && $2^{38}$&&          & 33:11:08\\
\bottomrule
\vspace{2pt}
\end{tabular}
\caption{$L$-series computations in genus 1 (elapsed times)}\label{table:Genus1FullTests}
\end{center}
\vspace{-12pt}
\begin{minipage}{1.0\linewidth}
\small
For the elliptic curve $y^2 = x^3+314159x+271828$, the coefficients of $L_p(T)$ were computed for all $p\le N$.  Columns labeled $\times n$ list total elapsed times (seconds or hh:mm:ss) for a computation performed on $n$ nodes (two cores per node), including communication overhead and time spent collating responses.  
\normalsize
\end{minipage}
\end{table}
\vspace{-12pt}
\begin{table}
\begin{center}
\begin{tabular}{@{}rcrrcrrcrr@{}}
& &\multicolumn{2}{c}{Genus 2}& &\multicolumn{2}{c}{Genus 3}& &\multicolumn{2}{c}{\!\!Genus 3 ($a_1$ only)}\\
\cmidrule(r){3-4}\cmidrule(r){6-7}\cmidrule(r){9-10}
$N$&$\qquad$& $\qquad\times 1$ & $\qquad\qquad\times 8$ &$\qquad$& $\qquad\times 1$& $\qquad\qquad\times 8$&
$\qquad$& $\qquad\times 1$& $\qquad\qquad\times 8$\\
\midrule
$2^{16}$& &        1 &       $<1$ & &       43 &       13 & &        1 &      $<1$\\
$2^{17}$& &        4 &        2 & &     1:49 &       18 & &        5 &        1\\
$2^{18}$& &       12 &        3 & &     4:42 &       41 & &       11 &        2\\
$2^{19}$& &       40 &        7 & &    12:43 &     1:47 & &       41 &        6\\
$2^{20}$& &     2:32 &       24 & &    36:14 &     4:52 & &     2:41 &       21\\
$2^{21}$& &    10:46 &     1:38 & &  1:45:36 &    13:40 & &    11:33 &     1:27\\
$2^{22}$& &    40:20 &     5:38 & &  5:23:31 &    41:07 & &    53:26 &     6:38\\
$2^{23}$& &  2:23:56 &    19:04 & & 16:38:11 &  2:05:40 & &  4:33:26 &    33:00\\
$2^{24}$& &  8:00:09 &  1:16:47 & &          &  6:28:25 & & 38:51:07 &  4:42:43\\
$2^{25}$& & 26:51:27 &  3:24:40 & &          & 20:35:16 & &          & 20:35:16\\
$2^{26}$& &          & 11:07:28 & &          &\\
$2^{27}$& &          & 36:48:52 & &          &\\
\bottomrule
\vspace{2pt}
\end{tabular}
\caption{$L$-series computations in genus 2 and 3 (elapsed times)}\label{table:FullTests}
\end{center}
\vspace{-12pt}
\begin{minipage}{1.0\linewidth}
\small
The coefficients of $L_p(T)$ were computed for the genus 2 and 3 hyperelliptic curves
\begin{align}\notag
y^2 &= x^5+31419x^3+271828x^2+1644934x+57721566;\\\notag
y^2 &= x^7+314159x^5+271828x^4+1644934x^3+57721566x^2+1618034x+141421,
\end{align}
for all $p \le N$ where the curves had good reduction.
Columns labeled $\times n$ list total elapsed wall times (hh:mm:ss) for a computation performed on $n$ nodes, including all overhead.  The last two columns give times to compute just the coefficient $a_1$.
\normalsize
\end{minipage}
\end{table}

\pagebreak
\bibliographystyle{amsplain}
\providecommand{\bysame}{\leavevmode\hbox to3em{\hrulefill}\thinspace}
\providecommand{\MR}{\relax\ifhmode\unskip\space\fi MR }
\providecommand{\MRhref}[2]{%
  \href{http://www.ams.org/mathscinet-getitem?mr=#1}{#2}
}
\providecommand{\href}[2]{#2}

\end{document}